\documentclass[12pt,a4paper]{amsart}
\usepackage[latin1]{inputenc}
\usepackage{graphicx}
\usepackage{amssymb, amsmath, mathrsfs}
\usepackage{geometry}
\usepackage{enumerate}
\usepackage[colorlinks=true,linkcolor=blue,urlcolor=blue,citecolor=blue]{hyperref}

\input xy
\xyoption{all}

\newtheorem{theorem}{Theorem}
\newtheorem{definition}[theorem]{Definition}
\newtheorem{lemma}[theorem]{Lemma}
\newtheorem{corollary}[theorem]{Corollary}
\newtheorem{proposition}[theorem]{Proposition}

\theoremstyle{definition}
\newtheorem{example}[theorem]{Example}
\newtheorem{remark}[theorem]{Remark}

\par
\begin{document}

\title[Strictly singular non-compact operators between $L_p$ spaces]{Strictly singular non-compact operators between $L_p$ spaces}

\author[F. L. Hern\'{a}ndez]{Francisco L. Hern\'{a}ndez}
\address{F. L. Hern\'{a}ndez\\ IMI \& Departamento de An\'{a}lisis Matem\'{a}tico y Matem\'atica Aplicada, Universidad Complutense de Madrid, 28040, Madrid, Spain.}
\email{pacoh@ucm.es}

\author[E. M. Semenov]{Evgeny M. Semenov}
\address{E. M. Semenov\\ Department of Mathematics, Voronezh State University, Voronezh 394006 (Russia).}
\email{nadezhka\_ssm@geophys.vsu.ru}

\author[P. Tradacete]{Pedro Tradacete}
\address{P. Tradacete\\ Instituto de Ciencias Matem\'aticas (CSIC-UAM-UC3M-UCM)\\
Consejo Superior de Investigaciones Cient\'ificas\\
C/ Nicol\'as Cabrera, 13--15, Campus de Cantoblanco UAM\\
28049 Madrid, Spain.}
\email{pedro.tradacete@icmat.es}

\thanks{Research partially supported by Agencia Estatal de Investigaci\'on (AEI) and Fondo Europeo de Desarrollo Regional (FEDER) through grants MTM2016-76808-P (AEI/FEDER, UE) and MTM2016-75196-P(AEI/FEDER, UE) as well as Grupo UCM 910346. E. Semenov was supported by Russian grant RFBR 18-01-00414. P. Tradacete gratefully acknowledges support of Spanish Ministerio de Econom\'{\i}a, Industria y Competitividad through the ``Severo Ochoa Programme for Centres of Excellence in R\&D'' (SEV-2015-0554).}




\maketitle

\begin{abstract}
We study the structure of strictly singular non-compact operators between $L_p$ spaces. Answering a question raised in \cite{HST17}, it is shown that there exist operators $T$, for which the set of points $(\frac1p,\frac1q)\in(0,1)\times (0,1)$ such that $T:L_p\rightarrow L_q$ is strictly singular but not compact contains a line segment in the triangle $\{(\frac1p,\frac1q):1<p<q<\infty\}$ of any positive slope. This will be achieved by means of Riesz potential operators between metric measure spaces with different Hausdorff dimension. The relation between compactness and strict singularity of regular operators defined on subspaces of $L_p$ is also explored.
\end{abstract}

\section{Introduction}

The purpose of this paper is to continue our recent research in \cite{HST17} about the relation between strict singularity and compactness for operators defined on the scale of $L_p$ spaces, and in particular, answer a question concerning the shape of the so-called $V$-characteristic sets, which consists of those points $(\frac1p,\frac1q)\in(0,1)\times (0,1)$ such that an operator $T:L_p\rightarrow L_q$ is strictly singular but not compact.

Recall that an operator between Banach spaces is \emph{strictly singular} provided it is not invertible when restricted to any (closed) infinite dimensional subspace. The class of strictly singular operators forms a closed two-sided operator ideal, containing that of compact operators, and was introduced by T. Kato \cite{Kato} in connection with the perturbation theory of Fredholm operators. Although strict singularity is a purely infinite-dimensional notion, the spectral theory for this class of operators coincides with that for compact operators (cf. \cite{AB}). On the other hand, strictly singular operators exhibit in general a different behaviour concerning duality \cite{W} and interpolation properties \cite{Beucher,Heinrich}.

J. Calkin \cite{C} noted that on Hilbert spaces there is only one non-trivial closed ideal, hence for an operator $T:L_2\rightarrow L_2$ strict singularity is equivalent to compactness. This fact was later extended by H. R. Pitt for operators $T:\ell_p\rightarrow \ell_q$ for $1\leq q\leq p<\infty$ (cf. \cite[Proposition 2.c.3]{LT1}). Among the simplest examples of strictly singular non-compact operators one should mention the formal inclusion
$$
i_{p,q}:\ell_p\hookrightarrow \ell_q
$$
when $1\leq p<q\leq\infty$. For a general Banach space, exhibiting instances of strictly singular non-compact operators can be very non-trivial (cf. \cite{AS}).

In this paper, we will deal with operators defined on $L_p$ spaces over finite measure spaces. Given $1\leq p,q\leq \infty$, let us denote by $L(L_p,L_q)$ the space of bounded linear operators $T:L_p\rightarrow L_q$, and $K(L_p,L_q)$ (respectively, $S(L_p,L_q)$ ) the ideal of compact (respectively, strictly singular) operators. For an operator $T:L_\infty\rightarrow L_1$, let us consider the characteristic sets
\begin{eqnarray*}
L(T)=\Big\{\Big(\frac1p,\frac1q\Big)\in(0,1)\times(0,1): T\in L(L_p,L_q)\Big\} & &\emph{($L$-characteristic)}, \\
K(T)=\Big\{\Big(\frac1p,\frac1q\Big)\in(0,1)\times(0,1): T\in K(L_p,L_q)\Big\} & &\emph{($K$-characteristic)}.
\end{eqnarray*}
These sets were introduced by M. A. Krasnoselskii and P. Zabreiko in \cite{ZK} and thoroughly analyzed in the monograph \cite{KZPS}. 

It is easy to see that $L(T)$ and $K(T)$ are monotone sets, in the sense that if a point $(\alpha_0,\beta_0)$ belongs to the set, then the upper-left corner
$$
\{(\alpha,\beta):0<\alpha\leq \alpha_0, \beta_0\leq \beta< 1\}
$$
is also contained in the set.

The classical Riesz-Thorin interpolation theorem tells us that $L(T)$ is a convex set, while Krasnoselskii's interpolation theorem \cite{Krasnoselskii} yields that $K(T)$ is convex as well. S. Riemenschneider showed in  \cite{R} that $L(T)$ is always an $F_\sigma$ set, and in fact charaterized those sets arising as the characteristic set $L(T)$ for some operator $T$ as precisely the convex, monotone $F_\sigma$ subsets of $(0,1)\times(0,1)$.

 In \cite{HST17}, motivated by the study of the interpolation properties of strictly singular operators on $L_p$ spaces, we focused on the \emph{$S$-characteristic set}
$$
S(T)=\Big\{\Big(\frac1p,\frac1q\Big)\in(0,1)\times(0,1): T\in S(L_p,L_q)\Big\},
$$
and in particular, on the set
$$
V(T)=S(T)\backslash K(T),
$$
which we will call the \emph{$V$-characteristic set} of the operator $T$. Note that $S$-characteristic sets are also monotone and convex (this follows from \cite[Theorem 21]{HST17}, see also Theorem \ref{teo:SSint} below).

For all known examples where $V(T)$ had been described, it consisted of (possibly degenerate) line segments in $(0,1)\times(0,1)$ which were vertical, horizontal or parallel to the diagonal. It was left as an open question in \cite{HST17} whether this was always the case. One of the purposes of this paper is to answer this question in the negative by exhibiting examples of operators whose corresponding set $V(T)$ can contain a line segment with any positive slope. It is relevant to note that the operators $T$ considered in \cite{R}, which allow to construct all possible $L$-characteristic sets, have the property that $K(T)=L(T)$, and thus $V(T)=\emptyset$. The examples that will be used here will be constructed by means of certain Riesz potential integral operators acting on $L_p$ spaces over Ahlfors regular metric measure spaces of different Hausdorff dimensions.

The paper is organized as follows. After some preliminaries on strictly singular and compact operators on $L_p$ spaces and more general Banach lattices, together with some facts about interpolation properties and geometric measure theory, in Section \ref{s:Vcharacteristic}, we will construct examples of operators with a variety of $V$-characteristic sets. In particular, we will show that if the $V$-characteristic set intersects the upper triangle $\{(\alpha,\beta):0<\alpha< \beta<1\}$, then it must do so in a vertical or horizontal segment (see Proposition \ref{prop:V_uppertriangle}), and that for any line segment $\ell$ contained in the lower triangle $\{(\alpha,\beta):0< \beta<\alpha<1\}$, there is an operator whose $V$-characteristic set contains $\ell$ (Theorem \ref{mainline}). Finally, motivated by the fact that positive strictly singular endomorphisms on $L_p$ are necessarily compact \cite{CG}, in the final section of the paper we will consider the case of regular operators defined on subspaces of $L_p$.

\section{Preliminaries}

Before we analyze the structure of $V$-characteristic sets, let us begin recalling some general facts related to strict singularity and compactness for operators on Banach lattices which will be helpful in our context.

Let $E$ be a Banach lattice, $X$ a Banach space,  and let $T:E\rightarrow X$ be an operator. We say $T$ is  \emph{AM-compact} when $T[-x,x]$ is a relatively compact set for every $x\in E_+$, where
$$
[-x,x]=\{y\in E:-x\leq y\leq x\}
$$
denotes the order interval generated by $x$. An operator $T:E\rightarrow X$ is called \emph{$M$-weakly compact} when $\|Tx_n\|\rightarrow0$ for every sequence $(x_n)$ of pairwise disjoint normalized vectors in $E$. Finally, we will say that an operator $T:E\rightarrow X$ is \emph{disjointly strictly singular} if for any sequence $(x_n)$ of pairwise disjoint vectors in $E$, the restriction
$$
T|_{[x_n]}:[x_n]\rightarrow X
$$
is not topologically invertible (here $[x_n]$ denotes the closed linear span of the sequence $(x_n)$).

It is well-known that an operator on a Banach lattice $T:E\rightarrow X$ is compact if and only if it is both AM-compact and $M$-weakly compact (cf. \cite[Proposition 3.7.4]{MN}). On the other hand, if $E$ is a Banach lattice  with finite cotype then an operator $T:E\rightarrow X$ is strictly singular if and only if it is both AM-compact and disjointly strictly singular \cite[Theorem 2.4]{FHKT}. In other words, in order to distinguish compactness from strict singularity on Banach lattices, one can reduce the focus to the behavior of pairwise disjoint sequences of vectors.

Throughout, unless otherwise stated, $L_p$ will denote $L_p(0,1)$ equipped with Lebesgue measure. For endomorphisms on $L_p$ spaces, the above characterization of strict singularity can also be expressed as follows: Given Banach spaces $X,Y,Z$, let us say that an operator $T:X\rightarrow Y$ is \emph{$Z$-singular} if it is never invertible when restricted to a subspace of $X$ isomorphic to $Z$; for $1< p<\infty$, an operator $T:L_p\rightarrow L_p$  is strictly singular if and only if it is both $\ell_2$-singular and $\ell_p$-singular \cite{M,Weis:77}.

To sum up this discussion, let us mention the following result given in \cite[Proposition 8]{HST17}: Suppose $T:L_p\rightarrow L_q$ for $2<q\leq p<\infty$ is strictly singular and not compact, then there exists a normalized sequence $(y_k)$ in $L_p$, which is equivalent to the unit vector basis of $\ell_2$, whose span $[y_k]$ is complemented, and such that $(Ty_k)$ is equivalent to the unit vector basis of $\ell_q$.

The fact that the inclusion $$I_{p,q}:L_p\hookrightarrow L_q$$ is continuous for $q<p$, together with the ideal property, yield that $L(T)$, $K(T)$ and $S(T)$ are monotone subsets of $(0,1)\times(0,1)$.

Convexity of $S(T)$ is equivalent to the interpolation property of strictly singular operators between $L_p$ spaces given in  \cite[Theorem 21]{HST17}. In order to recall this result, given $1\leq p_0,p_1,q_0,q_1\leq \infty$, for each $\theta\in(0,1)$ let
$$
\frac{1}{p_\theta}=\frac{\theta}{p_0}+\frac{1-\theta}{p_1}\,\,\textrm{ and }\,\,\frac{1}{q_\theta}=\frac{\theta}{q_0}+\frac{1-\theta}{q_1}.
$$

\begin{theorem}\label{teo:SSint}
Let $1\leq p_0,p_1,q_0,q_1<\infty$. If an operator $T:L_{p_0}\rightarrow L_{q_0}$ is strictly singular and $T:L_{p_1}\rightarrow L_{q_1}$ is bounded, then $T:L_{p_\theta}\rightarrow L_{q_\theta}$ is strictly singular for each $\theta\in(0,1)$.
\end{theorem}

Note that the assumption that the indices are finite in Theorem \ref{teo:SSint} is essential: Take the formal inclusion operator which satisfies that $T:L_\infty\rightarrow L_2$ is strictly singular (by Grothendieck's theorem), and $T:L_1\rightarrow L_1$ is bounded, but 
$$
T:L_p\rightarrow L_{\frac{2p}{p+1}}
$$ 
is not strictly singular for any $1<p<\infty$ (because of Kintchine's inequality).

Under some extra assumptions on the position of the interpolation segment
$$
\Big\{\theta\Big(\frac{1}{p_0},\frac{1}{q_0}\Big)+(1-\theta)\Big(\frac{1}{p_1},\frac{1}{q_1}\Big):\theta\in[0,1]\Big\}
$$
 in $(0,1)\times(0,1)$, one actually obtains a stronger property for the interpolated operator. We refer to this as a compact extrapolation result (see \cite[Theorem 13]{HST17}):

\begin{theorem}\label{teo:SSext}
Let $1<p_i,q_i<\infty$ for $i=0,1$ with  $q_0\neq q_1$ , $p_0\neq p_1$. Suppose either
\begin{itemize}
\item $\min\{\frac{q_0}{p_0},\frac{q_1}{p_1}\}\leq 1$, or
\item $\min\{\frac{q_0}{p_0},\frac{q_1}{p_1}\}> 1$ and $\frac{q_1-q_0}{p_1-p_0}<0$.
\end{itemize}
If $T$ is a bounded operator from $L_{p_i}$ to $L_{q_i}$ for $i=0,1$, and for some $0 < \theta <1 $, $T\in S(L_{p_{\theta}},L_{q_{\theta}})$, then $T\in K(L_{p_\tau},L_{q_\tau})$ for every $\tau\in(0,1)$.
\end{theorem}

These interpolation results can be useful in studying the rigidity of composition and Volterra-type operators in Hardy $H^p$ spaces (see \cite{LNST, MNST}).

It should be noted that $K(T)$ and $S(T)$ could be empty sets when $L(T)$ is not empty: take for instance the formal inclusion $J:L_\infty\hookrightarrow L_1$, for which
$$
L(J)=\Big\{\Big(\frac1p,\frac1q\Big):1<q\leq p<\infty\Big\}.
$$
More generally, if $g\in L_r$ and we define the multiplication operator $T_g:L_\infty\rightarrow L_1$ given by $T_g(f)=fg$, it is easy to check that
$$
L(T_g)=\Big\{\Big(\frac1p,\frac1q\Big):\frac1q\geq\frac1p+\frac1r\Big\},
$$
while
$$
K(T_g)=S(T_g)=\emptyset.
$$

In our previous work \cite[Theorem 18]{HST17}, it is shown that for $\alpha\in(0,1)$, the averaging operator of the form
$$
S_\alpha f=\sum_{k\in\mathbb N}\Big(\mu(A_k)^{\alpha-1}\int_{A_k} fd\mu\Big)\chi_{A_k},
$$
where $(A_k)_{k\in\mathbb N}$ is a sequence of pairwise disjoint measurable sets in $(0,1)$ with Lebesgue measure $\mu(A_k)>0$, satisfy that
$$
V(S_\alpha)=\Big\{\Big(\frac1p,\frac1q\Big):\frac{1}{q}=\frac{1}{p}-\alpha>0\Big\}.
$$

Similarly, for the Riemann-Liouville integral operator $R_\alpha$, for $0<\alpha<1$, defined by
$$
R_\alpha f(t)=\frac{1}{\Gamma(\alpha)}\int_0^t f(u)(t-u)^{\alpha-1} du
$$
it can be checked (cf. \cite{BL}). that $V(R_\alpha)$ coincides with a line segment parallel to the diagonal:
$$
V(R_\alpha)=\Big\{\Big(\frac1p,\frac1q\Big):\frac{1}{q}=\frac{1}{p}-\alpha>0\Big\}.
$$

In our search for operators with $V$-characteristic sets containing line segments which are not parallel to the diagonal we will need to consider a different approach.

As a consequence of the above results, one easily gets the following

\begin{corollary}\label{c:interiorequal}
If $S(T)\not= \emptyset$, then $S(T)$, $K(T)$ and $L(T)$ have the same interior.
\end{corollary}

In particular, we have
\begin{equation}\label{V(T)boundary}
V(T)\subseteq \partial L(T)
\end{equation}
(cf. \cite[Theorem 9]{HST17}), where $\partial L(T)$ denotes the boundary of $L(T)$ in $(0,1)\times (0,1)$. A similar fact holds in the particular case when $T$ is an integral operator, in which case one has $L(T)\backslash K(T)\subseteq\partial L(T)$ \cite[Theorem 5.14]{KZPS}.

The fact mentioned above that, for endomorphisms on $L_2$, compactness is the same as strict singularity, can be extended to operators $T:L_p\rightarrow L_q$ as long as $1<q\leq 2\leq p<\infty$ \cite[Theorem 5]{HST17}. In other words, we have that
\begin{equation}\label{V(T)notinsquare}
V(T)\cap \big\{(\alpha,\beta):0<\alpha\leq\frac12\leq \beta<1\big\}=\emptyset.
\end{equation}

Finally, let us recall the following symmetric property of the $V$-characteristic set. Let $\phi:(0,1)\times (0,1)\rightarrow (0,1)\times (0,1)$ be the involution given by
$$
\phi(\alpha,\beta)=(1-\beta,1-\alpha).
$$
This map is related to duality by the following property: for any operator $T:L_\infty\rightarrow L_1$ we have $(\alpha,\beta)\in L(T)$ if and only if $\phi(\alpha,\beta)\in L(T^*)$. Let
$$
R=\{(\alpha,\beta)\in(0,1)\times(0,1):\alpha<\beta<\frac12\}.
$$
From \cite[Theorem 7]{HST17}, one can deduce that
\begin{equation}\label{V(T)dual}
\phi(V(T)\cap R^c)\subset V(T^*).
\end{equation}

The following interpolation result due to E. Stein and G. Weiss (cf. \cite[Theorem IV.5.5]{BS}) will be particularly useful in the next section. Recall that for $1\leq p<\infty$, the space $L_{p,\infty}(\mu)$ consists of all measurable functions for which the following expression is finite
$$
\|f\|_{p,\infty}=\sup_{\lambda>0}\lambda\big[\mu(\{x:|f(x)|\geq\lambda\})\big]^{\frac1p}.
$$

\begin{theorem}\label{Stein-Weiss}
Let $1\leq p_0\leq q_0\leq\infty$, $1\leq p_1\leq q_1\leq\infty$, $p_0\neq p_1$, $q_0\neq q_1$ and $T:L_\infty(\nu)\rightarrow L_1(\mu)$  an operator such that for some $C_0,C_1>0$,
$$
\|T\chi_A\|_{q_i,\infty}\leq C_i \nu(A)^{\frac{1}{p_i}},
$$
for every measurable set $A$, and $i=0,1$. Then for each $\theta\in(0,1)$, the operator $T:L_{p_\theta}(\nu)\rightarrow L_{q_\theta}(\mu)$ is bounded.
\end{theorem}

Let us recall the definition of Hausdorff measure in a metric space $(X,d)$. For $A\subset X$, let $|A|=\sup\{d(x,y):x,y\in A\}$ denote the diameter of $A$. Given $s>0$, for $E\subset X$ and $\delta>0$, let
\begin{equation}\label{eq:Hsdelta}
\mathcal H_\delta^s(E)=\inf\{\sum_{i=1}^\infty |A_i|^s: E\subset \bigcup_{i=1}^\infty A_i,\,0<|A_i|\leq \delta\}.
\end{equation}
Let us denote
\begin{equation}\label{eq:Hs}
\mathcal H^s(E)=\lim_{\delta\rightarrow0} \mathcal H^s_\delta(E)=\sup_{\delta>0} \mathcal H^s_\delta(E),
\end{equation}
which always exists, though it could be infinite, and which defines an outer measure. The restriction of $\mathcal H^s$ to the $\sigma$-algebra of $\mathcal H^s$-measurable subsets, which contains the Borel sets, is the \textit{Hausdorff $s$-dimensional measure} (cf. \cite[Chapter 1]{Falconer}). Moreover, the \textit{Hausdorff dimension} of $E$ is the unique number $dim_{\mathcal H}(E)$ such that
$$
\mathcal H^s(E)=
\left\{
\begin{array}{ccl}
 \infty &   & \text{if }0\leq s<dim_{\mathcal H}(E),  \\
  &   &   \\
 0 &   & \text{if }s>dim_{\mathcal H}(E).  
\end{array}
\right.
$$

Recall that given $Q>0$, a metric space $(X,d)$ is said to be Ahlfors Q-regular if its Hausdorff dimension equals $Q$ and if $\mathcal H^Q$ denotes the corresponding Hausdorff measure, then there are constants $c,C>0$ such that for every $x\in X$ and $r>0$ we have
\begin{equation}\label{Q-regular}
    c r^Q\leq \mathcal{H}^Q(B(x,r))\leq C r^Q,
\end{equation}
where $B(x,r)=\{y\in X:d(x,y)< r\}$. 

In particular, we will be using the fact that for every $0<\alpha<1$, there exist a subset $\Omega_\alpha\subset(0,1)$ which equipped with the $\alpha$-dimensional Hausdorff measure and the euclidean metric is an Ahlfors $\alpha$-regular space (cf. \cite[Section 8.3]{Falconer}). This fact has been recently extended in \cite{AMS}, where it is shown that  any Ahlfors $Q$-regular metric space $(X,d)$ contains for every $0<\alpha<Q$, a closed subset $Y_\alpha\subset X$ such that $(Y_\alpha,d)$ is Ahlfors $\alpha$-regular.

\begin{lemma}\label{separablemeas}
For every $0<\alpha<1$, $(\Omega_\alpha,\mathcal H^\alpha)$ is an atomless separable measure space.
\end{lemma}

\begin{proof}
For every $x\in \Omega_\alpha$, we have that 
$$
\mathcal H^\alpha(\{x\})\leq \inf_{r>0} \mathcal H^\alpha(B(x,r))\leq C\inf_{r>0}r^\alpha=0.
$$
Thus, since  the measure $H^{\alpha}$ is  regular Borel, it follows that $(\Omega_\alpha, \mathcal H^\alpha)$ contains no atoms.

Now, take a countable set $(x_k)_{k=1}^\infty\subset \Omega_\alpha$ which is dense in $\Omega_\alpha$. Consider $\Sigma_{\mathbb Q}$, the $\sigma$-algebra generated by closed balls (intervals) in $\Omega_\alpha$, with center in some $x_k$ and rational radius. We claim that for every $\mathcal H^\alpha$-measurable subset $E$ of $\Omega_\alpha$, there is $S\in\Sigma_{\mathbb Q}$ such that $E\subset S$ and $\mathcal H^\alpha(E)= \mathcal H^\alpha(S)$. Separability will follow. 

Indeed, note first that in the definition of Hausdorff measure, the infimum in \eqref{eq:Hsdelta} can be computed only with $A_i$ being closed/open convex sets, that is, real intervals intersected with $\Omega_\alpha$. We follow a similar approach as in \cite[Theorem 1.6]{Falconer}. Given an $\mathcal H^\alpha$-measurable subset $E$ of $\Omega_\alpha$, for $i\in\mathbb N$ we can choose a collection of open convex sets $(U_{ij})_{j\in\mathbb N}$ such that $|U_{ij}|<\frac1i$ for each $j\in\mathbb N$, $E\subset \bigcup_{j\in \mathbb N} U_{ij}$ and 
$$
\sum_{j\in\mathbb N}|U_{ij}|^\alpha<\mathcal H^\alpha_{\frac1i}(E)+\frac1i. 
$$
For each $i,j\in\mathbb N$ we can consider $x_{k(i,j)}$  and $r_{ij}\in\mathbb Q$ such that the closed ball $B(x_{k(i,j)},r_{ij})$ satisfies
$$
U_{ij}\subset B(x_{k(i,j)},r_{ij})\quad\text{and}\quad |B(x_{k(i,j)},r_{ij})|<|U_{ij}|+\frac{1}{(i2^j)^{\frac1\alpha}}.
$$
Let 
$$
S=\bigcap_{i\in\mathbb N}\bigcup_{j\in\mathbb N} B(x_{k(i,j)},r_{ij})\in \Sigma_{\mathbb Q}.
$$
It follows that $E\subset S$ and  $S\subset \bigcup_{j\in\mathbb N} B(x_{k(i,j)},r_{ij})$ with 
$$
r_{ij}<\frac12\Big(|U_{ij}|+\frac{1}{(i2^j)^{\frac1\alpha}}\Big)<\frac1{2i}+\frac{1}{2(2i)^{\frac1\alpha}}.
$$ 
Thus, if we set $\delta_i=\frac1i+\frac{1}{(2i)^{\frac1\alpha}}$ it follows that
$$
\mathcal H^\alpha_{\delta_i}(S)\leq \sum_{j\in\mathbb N} | B(x_{k(i,j)},r_{ij})|^\alpha< \sum_{j\in\mathbb N} \Big(|U_{ij}|+\frac{1}{(i2^j)^{\frac1\alpha}}\Big)^\alpha \leq \sum_{j\in\mathbb N} |U_{ij}|^\alpha+ \frac1i< \mathcal H^\alpha_{\frac1i}(E)+\frac2i.
$$
Letting $i\rightarrow \infty$ we get $\mathcal H^\alpha(S)= \mathcal H^\alpha(E)$, as claimed.
\end{proof}

\section{The $V$-characteristic set of an operator}\label{s:Vcharacteristic}

Let us see next that if the $V$-characteristic set intersects the upper triangle $\{(\alpha,\beta):0<\alpha< \beta<1\}$, then it must do so in a vertical or horizontal segment (or in other words, with slope in $\{0,\infty\}$).

\begin{proposition}\label{prop:V_uppertriangle} 
Let $T:L_\infty\rightarrow L_1$ be an operator.
\begin{enumerate}
\item If there is $\alpha_0<\beta_0\leq\frac12$ such that $(\alpha_0,\beta_0)\in V(T)$, then
$$
\{(\alpha,\beta_0):0<\alpha\leq \alpha_0\}\subset V(T).
$$
\item If there is $\frac12\leq\alpha_0<\beta_0$ such that $(\alpha_0,\beta_0)\in V(T)$, then
$$
\{(\alpha_0,\beta):\beta_0\leq\beta< 1\}\subset V(T).
$$
\end{enumerate}
\end{proposition}

\begin{proof}
Suppose that $\alpha_0<\beta_0\leq\frac12$ and let $p=\frac1{\alpha_0}$, $q=\frac1{\beta_0}$, so that $2\leq q<p<\infty$. Note that for this range of $p$ and $q$, every operator $T:L_p\rightarrow L_q$ is M-weakly compact. Indeed, suppose otherwise that there exists a disjoint sequence $(x_n)$ in the ball of $L_p$ such that $\|Tx_n\|_q\geq \alpha>0$. In particular, $(x_n)$ is equivalent to the unit vector basis of $\ell_p$. By \cite{KP}, there is a subsequence of $(x_n)$, not relabelled, such that $(Tx_n)$ is either equivalent to the unit vector basis of $\ell_q$ or $\ell_2$. In either case, we would have
$$
n^{1/q}\lesssim \Big\|\sum_{i=1}^n Tx_i\Big\|\leq\|T\|\Big\|\sum_{i=1}^n x_i\Big\|\lesssim n^{1/p},
$$
which is a contradiction for large $n$ with $q<p$.

Now if $(\alpha_0,\beta_0)\in V(T)$ for some $T:L_\infty\rightarrow L_1$, we have that $T:L_p\rightarrow L_q$ is strictly singular and not compact. Since strictly singular operators form an ideal, for every $r\in[p,\infty]$ it follows that
$$
T:L_r\hookrightarrow L_p\rightarrow L_q
$$
is strictly singular. Suppose that for some $r\in[p,\infty]$, the operator $T:L_r\rightarrow L_q$ were compact. We claim that in this case $T:L_p\rightarrow L_q$ must be AM-compact: indeed, if for some $r\in[p,\infty]$ $T:L_r\rightarrow L_q$ is compact, then so is $T:L_\infty\rightarrow L_q$. Now, for arbitrary $f\in L_p$, and any $\varepsilon>0$, taking $M_\varepsilon\in\mathbb R_+$ such that $\|(|f|-M_\varepsilon)_+\|_p\leq\varepsilon$ we have that
$$
[-|f|,|f|]\subset[-M_\varepsilon,M_\varepsilon]+\varepsilon B_{L_p}.
$$
Since $T[-M_\varepsilon,M_\varepsilon]$ is relatively compact, it follows that $T[-|f|,|f|]$ is also relatively compact in $L_q$. Thus, $T$ is AM-compact, as claimed.

Therefore, by \cite[Proposition 3.7.4]{MN}, $T:L_p\rightarrow L_q$ is compact, being AM-compact and M-weakly compact. This contradiction shows that for every $r>p$, $T:L_r\rightarrow L_q$ is not compact, so
$$
\{(\alpha,\beta_0):0<\alpha\leq \alpha_0\}=\Big\{\Big(\frac1r,\frac1q\Big):r\geq p\Big\}\subset V(T).
$$

Finally, the case when $\frac12\leq\alpha_0<\beta_0$  follows from the previous one using duality arguments together with \eqref{V(T)dual}.
\end{proof}

\begin{remark}
As a consequence of \cite[Theorem 9]{HST17}, in part (1) above it holds that $T$ is not bounded from $L_r$ to $L_s$ for any $1<r<\infty$, $s>q$; while in part (2), $T$ is not bounded from $L_r$ to $L_s$ for any $1<s<\infty$ and $r<p$.
\end{remark}

\begin{definition}
Let us denote by $\mathfrak{L}$ the set of all affine lines $\ell$ with slope
$k(\ell)\in[0,\infty]$, intersected with the square $(0,1)\times (0,1)$ such that either:
\begin{enumerate}
\item $k(\ell)=0,\infty$ and $\ell\cap \{(\alpha,\beta):0<\alpha\leq\frac12\leq\beta<1\}=\emptyset;$ or,
\item $k(\ell)>0$ and $\ell$ lies entirely below the diagonal $\{(\alpha,\alpha):\alpha\in(0,1)\}$.
\end{enumerate}
\end{definition}

Note each $\ell\in\mathfrak{L}$ decomposes the square $(0,1)\times(0,1)$ in three disjoint regions $L_\ell$, $R_\ell$ and $\ell$, where $L_\ell$ denotes the left-hand side of $\ell$, that is, the one containing the set $\{(\alpha,\beta):0<\alpha\leq\frac12\leq\beta<1\}$.

\begin{theorem}\label{mainline}
For each $\ell\in\mathfrak{L}$, there is an operator $T:L_\infty\rightarrow L_1$ such that
\begin{enumerate}
    \item $L(T)=S(T)=L_\ell\cup \ell$,
    \item $K(T)=L_\ell$,
    \item $V(T)=\ell$.
\end{enumerate}
\end{theorem}

For the proof of this result, we will make use of \emph{Riesz potential} integral operators between different measure spaces. We refer to \cite[Chapter 8]{KZPS} for background and details on these operators.

Consider the real segment $(0,1)$ equipped with its standard metric and Lebesgue measure. For each Borel set $A\subset (0,1)$, let $\mu(A)$ denote its Lebesgue measure. Let $0<\alpha<1$. According to \cite{AMS} we can take $\Omega_\alpha\subset (0,1)$ a closed Ahlfors $\alpha$-regular subset and denote $\mathcal H^\alpha$ for the corresponding $\alpha$-dimensional Hausdorff measure. Given $0<\lambda<1$, we consider the Riesz potential integral operator $T_\lambda:L_\infty(0,1)\rightarrow L_1(\Omega_\alpha,\mathcal H^\alpha)$  defined by
\begin{equation}\label{Rieszpotential}
T_\lambda f(t)=\int_0^1 \frac{f(u)}{|t-u|^\lambda} du,\quad t\in\Omega_\alpha\subset (0,1).
\end{equation}

\begin{theorem}\label{teo:compactRiesz}
Let $0<\lambda,\alpha<1$. The operator $T_\lambda:L_\infty(0,1)\rightarrow L_1(\Omega_\alpha)$ given in \eqref{Rieszpotential} satisfies that
$$
V(T_\lambda)=\Big\{\Big(x\,,\,\frac{1}{\alpha}(x-1+\lambda)\Big):1-\lambda<x<\min\Big\{1,\frac{1-\lambda}{1-\alpha}\Big\}\Big\}.
$$
\end{theorem}

\begin{proof}
The proof follows the ideas of \cite[Section 8]{KZPS}. For the sake of simplicity, let $T=T_\lambda$. In particular, we have that the argument in the proof of \cite[Theorem 8.3]{KZPS} yields that for some constant $C_\lambda>0$ and every Borel set $A\subset (0,1)$ we have
\begin{equation}\label{rw(1-lambda,0)}
|T \chi_A(t)|\leq C_\lambda \mu(A)^{1-\lambda}.
\end{equation}

Suppose that $\alpha<\lambda$ (in case we have the converse inequality, the argument will be similar). Given $x\in (1-\lambda,\frac{1-\lambda}{1-\alpha})$, let $q_1=\frac{\alpha}{x-1+\lambda}$. We claim that there is $C_x>0$ such that for every Borel set $A\subset (0,1)$ we have
\begin{equation}\label{re(x,(x-1+lambda)/s)}
\|T \chi_A\|_{q_1,\infty}\leq C_x \mu(A)^x.
\end{equation}

To prove this inequality, note first that as $x<\frac{1-\lambda}{1-\alpha}$, we have $\frac{1}{\alpha}(x-1+\lambda)<x$. Hence, we can take $\theta\in (\frac1\alpha(x-1+\lambda),x)$, and applying Holder's inequality we have
\begin{align*}
|T \chi_A(t)|&=\int_A |t-u|^{-\lambda}du=\int_A |t-u|^{1-x} |t-u|^{-(x-1+\lambda)}du\\
&\leq  \Big(\int_A |t-u|^{\frac{1-x}{1-\theta}}du\Big)^{1-\theta} \Big(\int_A |t-u|^{-\frac{x-1+\lambda}{\theta}}du\Big)^{\theta}.
\end{align*}

Now, arguing again as in the proof of \cite[Theorem 8.3]{KZPS}, for some $C>0$, if $r=\mu(A)/2$ we have that
\begin{equation*}
\int_A |t-u|^{\frac{1-x}{1-\theta}}du\leq \int_{B(t,r)} |t-u|^{\frac{1-x}{1-\theta}}du=C\mu(A)^{\frac{x-\theta}{1-\theta}}.
\end{equation*}
Therefore, we get
\begin{equation}\label{pointwiseT_lambda}
|T \chi_A(t)|\leq C^{1-\theta} \mu(A)^{x-\theta} \Big(\int_A |t-u|^{-\frac{x-1+\lambda}{\theta}}du\Big)^{\theta}.
\end{equation}

For each $u\in A$, let $g_u(t)=|t-u|^{-\frac{x-1+\lambda}{\theta}}$ for $t\in \Omega_\alpha$. Then, we have that
\begin{align}\label{gu}
\|g_u\|_{\theta q_1,\infty}&=\sup_{h>0}h\big[\mathcal H^\alpha(\{t\in \Omega_\alpha:g_u(t)\geq h\})\big]^{\frac{1}{\theta q_1}}\\
\nonumber &= \sup_{h>0}h\big[\mathcal H^\alpha(B(u,h^{-\frac{\theta}{x-1+\lambda}}))\big]^{\frac{1}{\theta q_1}}\\
\nonumber &\leq C_\alpha,
\end{align}
where $C_\alpha$ is a constant arising from the fact that $\mathcal H^\alpha$ is Ahlfors $\alpha$-regular.

Note that, due to our choice of $\theta$, it follows that $\theta q_1>1$, so in particular the expression $\|\cdot\|_{\theta q_1,\infty}$ is equivalent to a norm (cf. \cite[Lemma IV.4.5]{BS}). This fact, together with \eqref{pointwiseT_lambda} and \eqref{gu}, implies that
\begin{align*}
\|T \chi_A\|_{q_1,\infty}&\leq C^{1-\theta}\mu(A)^{x-\theta} \Big\|\Big(\int_A |t-u|^{-\frac{x-1+\lambda}{\theta}}du\Big)^{\theta}\Big\|_{q_1,\infty}\\
&=C^{1-\theta}\mu(A)^{x-\theta} \Big\|\int_A |t-u|^{-\frac{x-1+\lambda}{\theta}}du\Big\|^{\theta}_{\theta q_1,\infty}\\
&\leq C^{1-\theta}\mu(A)^{x-\theta} \Big(\int_A \|g_u\|_{\theta q_1,\infty}du\Big)^{\theta}\\
&\leq C^{1-\theta} C_\alpha^\theta\mu(A)^x,
\end{align*}
as we wanted to show.

Therefore, putting together \eqref{rw(1-lambda,0)} and \eqref{re(x,(x-1+lambda)/s)}, we can apply Theorem \ref{Stein-Weiss} with $p_0=\frac{1}{1-\lambda}$, $q_0=\infty$, $p_1=\frac{1}{x}$ and $q_1=\frac{\alpha}{x-1+\lambda}$, to conclude that $T: L_p(0,1)\rightarrow L_q(\Omega_\alpha)$ is bounded for every
$$
1-\lambda<\frac1p<\frac{1-\lambda}{1-\alpha},\quad\quad\frac1q=\frac{1}{\alpha}\Big(\frac1p-1+\lambda\Big).
$$ 

Moreover, since $p\neq q$ and $T$ is a positive integral operator, \cite[Proposition 2.6]{Flores} yields that $T \in S(L_p(0,1),L_q(\Omega_\alpha))$.

Finally, let us see that $T\notin K(L_p(0,1),L_q(\Omega_\alpha))$. Let $t_0\in \Omega_\alpha$. For large enough $k_0\in\mathbb N$, we have $B_k=(t_0-\frac{1}{2k},t_0+\frac{1}{2k})\subset (0,1)$ for every $k\geq k_0$. Let
$$
f_k=\frac{\chi_{B_k}}{\mu(B_k)^{\frac1p}}=k^{\frac1p}\chi_{B_k}.
$$
Then, for $t\in\Omega_s$ we have
$$
T f_k(t)=k^{\frac1p} \int_{B_k} |t-u|^{-\lambda} du.
$$
Let us see that the set $(T_\lambda f_k)_{k= k_0}^\infty$ is not uniformly $q$-integrable. Indeed, let $B'_k=B_k\cap \Omega_\alpha$. For $u\in B_k$ and $t\in B'_k$, we have $|t-u|\leq\frac1{k}$, hence
$$
T f_k(t)\geq k^{\lambda+\frac1p-1}.
$$
Therefore, using that $\mathcal H^\alpha$ is Ahlfors $\alpha$-regular there exists $c_\alpha>0$ such that
$$
\|(T_\lambda f_k)\chi_{B'_k}\|_q\geq \| k^{\lambda+\frac1p-1}\chi_{B'_k}\|_q= k^{\lambda+\frac1p-1}\mathcal H^\alpha(B'_k)^{\frac1q}\geq c_\alpha>0.
$$
Since this holds for every $k\geq k_0$, it follows that $(T f_k)_{k= k_0}^\infty$ is not uniformly $q$-integrable, and so $T \notin K(L_p(0,1),L_q(\Omega_\alpha))$ as claimed (cf. \cite[p. 49]{KZPS}).
\end{proof}

\begin{remark}
The above proof leaves open the question whether the operator $T_\lambda: L_{\frac{1-s}{1-\lambda}}(0,1)\rightarrow L_{\frac{1-\alpha}{1-\lambda}}(\Omega_\alpha)$ is actually bounded. Even in case it were, we do not know if it would be strictly singular. However, it can be deduced from Proposition \ref{prop:V_uppertriangle}, that $T: L_p(0,1)\rightarrow L_p(\Omega_\alpha)$ cannot be bounded for any $p<{\frac{1-\alpha}{1-\lambda}}$.
\end{remark}

\begin{figure}
\centering
\ifx\JPicScale\undefined\def\JPicScale{0.5}\fi
\unitlength \JPicScale mm
\begin{picture}(110,110)(0,0)
\linethickness{0.3mm}
\put(-0.2,0){\line(1,0){100.1}}
\linethickness{0.3mm}
\put(0,0){\line(0,1){100}}
\linethickness{0.1mm}
\put(0,0){\line(1,1){51.3}}
\linethickness{0.1mm}
\put(53,53){\line(1,1){27}}
\linethickness{0.3mm}
\put(80,0){\line(0,1){80}}
\put(0,80){\line(1,0){80.2}}
\linethickness{0.3mm}
\linethickness{1mm}
\put(35,0){\line(1,3){17}}
\put(35,-5){\makebox(0,0)[cc]{\tiny{$(\!1\!-\!\lambda,\!0)$}}}
\put(-7,80){\makebox(0,0)[cc]{\tiny{$(\!0,\!1\!)$}}}
\put(52.2,52.2){\circle{2}}
\put(36,55){\makebox(0,0)[cc]{\tiny{$\!(\!\frac{1-\lambda}{1-\alpha},\!\frac{1-\lambda}{1-\alpha}\!)$}}}
\put(98,0){\makebox(0,0)[cc]{$>$}}
\put(105,0){\makebox(0,0)[cc]{\tiny{$\frac1p$}}}
\put(0,98){\makebox(0,0)[cc]{$\wedge$}}
\put(0,107){\makebox(0,0)[cc]{\tiny{$\frac1q$}}}
\end{picture}
\caption{The $V$-characteristic set of the operator $T_\lambda$ in Theorem \ref{teo:compactRiesz}.\label{figureVT}}
\end{figure}
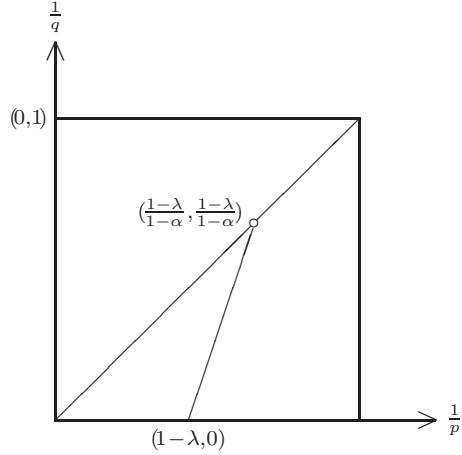

\begin{proof}[Proof of Theorem \ref{mainline}]
We will consider four separate cases depending on the slope $k(\ell)$ of the line segment $\ell\in\mathfrak L$.

\begin{enumerate}
\item $k(\ell)=0$: Since we must have
$$
\ell\cap \{(\alpha,\beta):0<\alpha\leq\frac12\leq\beta<1\}=\emptyset,
$$
then $\ell$ is the horizontal line of equation $\beta=\beta_0$ with $\beta_0<\frac12$. Let $q_0=\frac1{\beta_0}$, and consider the operator $T:L_p\rightarrow L_{q_0}$ given by
$$
T(f)=\sum_{n=1}^\infty\int_0^1 f(t) r_n (t)dt\, f_n,
$$
where $(r_n)$ denote the Rademacher functions and $(f_n)$ is any sequence of normalized pairwise disjoint functions in $L_{q_0}$. It is clear that $T$ admits the factorization
$$
\xymatrix{L_p\ar_{P_{rad}}[d]\ar^T[rr]&&L_{q_0}\\
\ell_2\ar@{^{(}->}_{i_{2,q_0}}[rr]&&\ell_{q_0} \ar_{J}[u] }
$$
where $P_{rad}$ is the projection onto the closed linear span of the Rademacher functions, and $J$ the isometric embeddings via the sequence $(f_n)$ of disjointly supported functions in $L_{q_0}$. It follows easily that $T$ has the required properties.

\item $k(\ell)=\infty$: In this case, we must have that $\ell$ is the vertical line of equation $\alpha=\alpha_0$ with $\alpha_0>\frac12$. Let $p_0=\frac1{\alpha_0}$ and consider the operator $T:L_{p_0}\rightarrow L_q$ given by
$$
T(f)=\sum_{n=1}^\infty\int_0^1 f(t) g_n (t)dt \,r_n,
$$
where $(r_n)$ also denote the Rademacher functions and $(g_n)$ is any sequence of normalized pairwise disjoint functions in $L^*_{p_0}$. It is clear that $T$ admits the factorization
$$
\xymatrix{L_{p_0}\ar_{P}[d]\ar^T[rr]&&L_{q}\\
\ell_{p_0}\ar@{^{(}->}_{i_{p_0,2}}[rr]&&\ell_{2} \ar_{J_{rad}}[u] }
$$
where $P$ is a projection onto the closed linear span of a sequence $(f_n)$ of disjointly supported functions in $L_{p_0}$ satisfying $$\int_0^1f_ng_n=1,$$ and $J_{rad}$ is the embeddings via the Rademacher functions. It follows easily that $T$ has the required properties.

\item $1\leq k(\ell)<\infty$: Let $s=k(\ell)^{-1}\leq1$ and let $\Omega_s\subset (0,1)$ the Ahlfors $s$-regular subset considered in Theorem  \ref{teo:compactRiesz}. Note that by Carath\'eodory's isomorphism theorem (cf. \cite[Ch. 15, Theorem 2]{Royden}) and Lemma \ref{separablemeas},  $L_p(0,1)$ is lattice isometric to $L_p(\Omega_\alpha,\mathcal H^\alpha)$. Now, for $0<\lambda\leq s$ the Riesz potential operator
$$
T=T_\lambda:L_\infty(0,1)\rightarrow L_1(\Omega_\alpha)
$$
given in \eqref{Rieszpotential} satisfies the required properties by Theorem \ref{teo:compactRiesz}.

\item $0<k(\ell)<1$: Let $\ell'=\phi(\ell),$ denote the conjugate line segment to $\ell$. Since $k(\ell')>1$, we can apply the previous argument to construct an operator $T$ such that $V(T)=\ell'$ and $L(T)=L_{\ell'}\cup\ell'$. Hence, by the duality property given in \eqref{V(T)dual}, we get that the adjoint operator satisfies $V(T^*)=\ell$ and $L(T^*)=S(T^*)=L_\ell\cup\ell$, as required.
\end{enumerate}
\end{proof}

\begin{remark}
In the case when the slope of the line segment is rational and greater than one, say $\frac{n}{m}$ with $m\leq n$, one can consider the $n$-dimensional Riesz potential operator 
$$
T_\lambda f(t)=\int_\Omega \frac{f(s)}{\|t-s\|^\lambda} d\mu(s),\quad t\in\Omega,
$$
as an operator $T_\lambda: L_\infty(\Omega)\rightarrow L_1(\Omega')$. Here $\Omega$ denotes the euclidean unit ball in $\mathbb R^n$ centered at the origin, $0<\lambda<n$, $\|\cdot\|$ denotes the euclidean norm in $\mathbb R^n$, $\mu$ is the corresponding Lebesgue measure, and 
$$
\Omega'=\{(x,0)\in\mathbb R^m\times \mathbb R^{n-m}: \|x\|\leq1\}
$$ 
is the $m$-dimensional unit ball embedded in $\mathbb R^n$. In this case, Theorem \ref{mainline} can also be deduced with similar reasonings as above and using also \cite[Theorems 8.3 and 8.9]{KZPS}.
\end{remark}

Recall that if $T_1,T_2:L_\infty\rightarrow L_1$ are positive operators, then
$$
L(T_1+T_2)=L(T_1)\cap L(T_2).
$$

Let us note that the operators exhibited in the proof of Theorem \ref{mainline} for $k(\ell)\in(0,\infty)$ are instances of positive integral operators. In particular, these can be combined to construct operators with a more elaborate $V$-characteristic set as follows.

\begin{proposition}\label{p:intersection}
Let $(T_n)_{n\in\mathbb N}\subset L(L_\infty,L_1)$ be a sequence of positive operators such that
$$
\sup\Big\{\|T_n\|_{L(L_p,L_q)}:n\in\mathbb N,\big(\frac1p,\frac1q\big)\in \bigcap_{k\in\mathbb N} L(T_k)\Big\}<\infty.
$$
Then the operator $T=\sum_{n\in\mathbb N} 2^{-n}T_n$ satisfies
$$
V(T)=\bigcup_{n\in\mathbb N} V(T_n)\cap\bigcap_{n\in\mathbb N} S(T_n).
$$
\end{proposition}

\begin{proof}
Let $1<p,q<\infty$. First note that the series $\sum_n 2^{-n}T_n$ converges to $T$ in the norm of $L(L_p,L_q)$ for each
$$
\Big(\frac1p,\frac1q\Big)\in \bigcap_{k\in\mathbb N} L(T_k).
$$

Suppose first that $(\frac1p,\frac1q)\in V(T)$, or equivalently, $T\in S(L_p,L_q)\backslash K(L_p,L_q)$. In this case, as $0\leq T_n \leq T:L_p\rightarrow L_q$, the domination property of strictly singular operators (see \cite{FH}) yields that $T_n\in S(L_p,L_q)$ for every $n\in\mathbb N$. 

On the othe hand, if $T_n\in K(L_p,L_q)$ for every $n\in\mathbb N$, then so would be $T$. Hence, there must be some $n\in\mathbb N$ such that $T_n\in V_{p,q}$.

Conversely, let us assume that the point
$$
\Big(\frac1p,\frac1q\Big)\in\bigcup_{n\in\mathbb N} V(T_n)\cap\bigcap_{n\in\mathbb N} S(T_n).
$$
In particular, we have that $T_n\in S(L_p,L_q)$ for every $n\in\mathbb N$, which implies that $T\in S(L_p,L_q)$. Suppose $T\in K(L_p,L_q)$, then by the domination property of compact operators (see \cite{DF}) we would have that for every $n\in\mathbb N$, $T_n\in K(L_p,L_q)$. This is a contradiction which finishes the proof.
\end{proof}

From Theorem \ref{mainline} and Proposition \ref{p:intersection}, we get the following

\begin{corollary}\label{cor:polygon}
Given line segments $\ell_1,\ell_2,\ldots,\ell_n$ in $\mathfrak L$, with slopes $0< k(\ell_1)<k(\ell_2)<\ldots<k(\ell_n)< \infty$, the set $$P=\bigcap_{i=1}^n L_{\ell_i}$$ is a monotone convex polygon of $(0,1)\times(0,1)$, and there is a positive integral operator $T:L_\infty\rightarrow L_1$ such that $V(T)=\partial P$ (the boundary of $P$ relative to $(0,1)\times(0,1)$).
\end{corollary}

\begin{proof}
By Theorem \ref{mainline}, for every $1\leq i\leq n$ we can consider an operator $T_i:L_\infty\rightarrow L_1$ such that $V(T_i)=\ell_i$. Let $T=\sum_{i=1}^n 2^{-i}T_i$.The conclusion follows from Proposition  \ref{p:intersection}.
\end{proof}

A similar argument can be used to build $V$-characteristic sets for polygons with horizontal and vertical segments, the only difference is that we must also consider non-integral operators:

\begin{corollary}\label{cor:polygon2}
Given line segments $\ell_1,\ell_2,\ldots,\ell_n$ in $\mathfrak L$, with slopes $0= k(\ell_1)<k(\ell_2)<\ldots<k(\ell_n)= \infty$, the set $$P=\bigcap_{i=1}^n L_{\ell_i}$$ is a monotone convex polygon of $(0,1)\times(0,1)$, and there is an operator $T:L_\infty\rightarrow L_1$ such that $V(T)=\partial P$.
\end{corollary}

\begin{proof}
We proceed as in \cite[Example 20]{HST17}, decomposing the underlying measure space in three disjoint parts and considering, for the horizontal and vertical segments, the operators $T_1$ and $T_n$ as defined in the proof of Theorem \ref{mainline}, cases (1) and (2).
\end{proof}

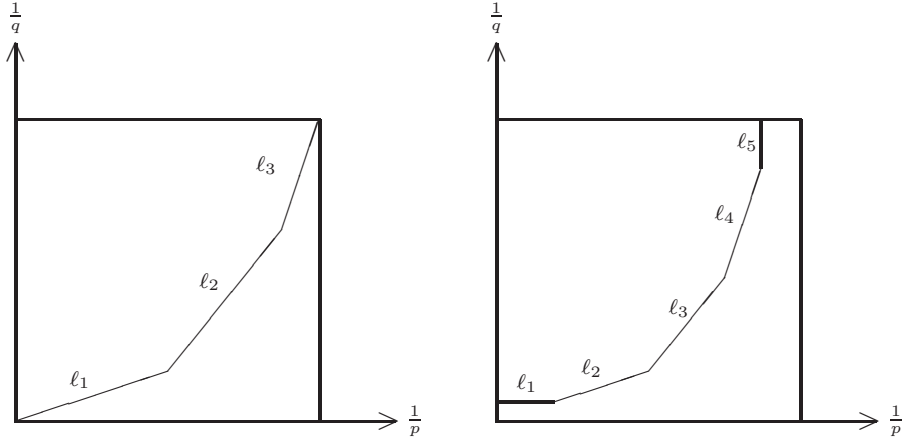
\begin{figure}
\centering
\ifx\JPicScale\undefined\def\JPicScale{0.5}\fi
\unitlength \JPicScale mm
\begin{picture}(110,110)(0,0)
\linethickness{0.3mm}
\put(-0.2,0){\line(1,0){100.1}}
\put(0,0){\line(0,1){100}}
\put(80,0){\line(0,1){80}}
\put(0,80){\line(1,0){80.2}}
\put(0,0){\line(3,1){40}}
\put(40,13.3){\line(4,5){30}}
\put(70,51){\line(1,3){9.7}}
\linethickness{0.3mm}
\put(98,0){\makebox(0,0)[cc]{$>$}}
\put(105,0){\makebox(0,0)[cc]{\tiny{$\frac1p$}}}
\put(0,98){\makebox(0,0)[cc]{$\wedge$}}
\put(0,107){\makebox(0,0)[cc]{\tiny{$\frac1q$}}}
\put(17,11){\makebox(0,0)[cc]{\tiny{$\ell_1$}}}
\put(51,37){\makebox(0,0)[cc]{\tiny{$\ell_2$}}}
\put(66,68){\makebox(0,0)[cc]{\tiny{$\ell_3$}}}
\end{picture}\quad\quad
\begin{picture}(110,110)(0,0)
\linethickness{0.3mm}
\put(-0.2,0){\line(1,0){100.1}}
\put(0,0){\line(0,1){100}}
\put(80,0){\line(0,1){80}}
\put(0,80){\line(1,0){80.2}}
\put(15,5){\line(3,1){25}}
\put(0,5){\line(1,0){15}}
\put(40,13.3){\line(4,5){20}}
\put(59.9,38.1){\line(1,3){9.7}}
\put(69.5,67){\line(0,1){13}}
\linethickness{0.3mm}
\put(98,0){\makebox(0,0)[cc]{$>$}}
\put(105,0){\makebox(0,0)[cc]{\tiny{$\frac1p$}}}
\put(0,98){\makebox(0,0)[cc]{$\wedge$}}
\put(0,107){\makebox(0,0)[cc]{\tiny{$\frac1q$}}}
\put(8,10){\makebox(0,0)[cc]{\tiny{$\ell_1$}}}
\put(25,14){\makebox(0,0)[cc]{\tiny{$\ell_2$}}}
\put(48,30){\makebox(0,0)[cc]{\tiny{$\ell_3$}}}
\put(60,55){\makebox(0,0)[cc]{\tiny{$\ell_4$}}}
\put(66,74){\makebox(0,0)[cc]{\tiny{$\ell_5$}}}
\end{picture}

\caption{The $V$-characteristic sets of the operators in Corollaries \ref{cor:polygon} and \ref{cor:polygon2}.\label{figureVT}}
\end{figure}

\begin{example}\label{curve}
Given any function $\varphi:(0,1)\rightarrow(0,1)$, let
\begin{align*}
E_\varphi&=\{(s,t):0<s<1,\varphi(s)\leq t<1\},\\
C_\varphi&=\{(t,\varphi(t)):t\in(0,1)\}.
\end{align*}
If $\varphi:(0,1)\rightarrow(0,1)$ is a convex piecewise differentiable function such that $\varphi(t)<t$, then there is a positive integral operator $T:L_\infty\rightarrow L_1$ so that
\begin{enumerate}
\item $L(T)=E_\varphi$.
\item $S(T)=E_\varphi$.
\item $V(T)$ is a dense subset of the curve $C_\varphi$.
\end{enumerate}

Indeed, we can take $(x_n)$ a dense set of $(0,1)$ such that $\varphi'(x_n)$ is well defined and positive. For each $n\in\mathbb N$, let $\ell_n$ be the line segment on $(0,1)\times (0,1)$ which goes through the point $(x_n,\varphi(x_n))$ with slope $k(\ell_n)=\varphi'(x_n)$. By Theorem \ref{mainline}, there is an operator $T_n:L_\infty\rightarrow L_1$ such that $L(T_n)=S(T_n)=L_{\ell_n}\cup\ell_n$ and $V(T_n)=\ell_n$. From convexity it follows that
$$
E_\varphi =\bigcap_{n\in\mathbb N} L_{\ell_n}\cup\ell_n.
$$
Taking $T=\sum_{n\in\mathbb N}2^{-n}T_n$, we have that
$$
L(T)=\bigcap_{n\in\mathbb N} L(T_n)=E_\varphi.
$$
Also, since $T$ is a positive integral operator, and $C_\varphi\cap\{(t,t):t\in(0,1)\}=\emptyset$, \cite[Proposition 2.6]{Flores} yields that $C_\varphi\subset S(T)$, so by Corollary \ref{c:interiorequal} we get
$$
S(T)=L(T).
$$
Finally, Proposition \ref{p:intersection} yields that
$$
V(T)=\Big(\bigcup_{n\in\mathbb N}\ell_n\Big)\cap E_\varphi=\{(x_n,\varphi(x_n)):n\in\mathbb N\}\subset C_\varphi.
$$
\end{example}

\begin{figure}
\centering
\ifx\JPicScale\undefined\def\JPicScale{0.5}\fi
\unitlength \JPicScale mm
\begin{picture}(110,110)(0,0)
\linethickness{0.3mm}
\put(-0.2,0){\line(1,0){100.1}}
\put(0,0){\line(0,1){100}}
\put(80,0){\line(0,1){80}}
\put(0,80){\line(1,0){80.2}}
\thinlines
\qbezier(0,0)(60,10)(80,80)
\put(98,0){\makebox(0,0)[cc]{$>$}}
\put(105,0){\makebox(0,0)[cc]{\tiny{$\frac1p$}}}
\put(0,98){\makebox(0,0)[cc]{$\wedge$}}
\put(0,107){\makebox(0,0)[cc]{\tiny{$\frac1q$}}}
\end{picture}

\caption{The $V$-characteristic set of an operator as in Example \ref{curve}.\label{figureVTcurve}}
\end{figure}
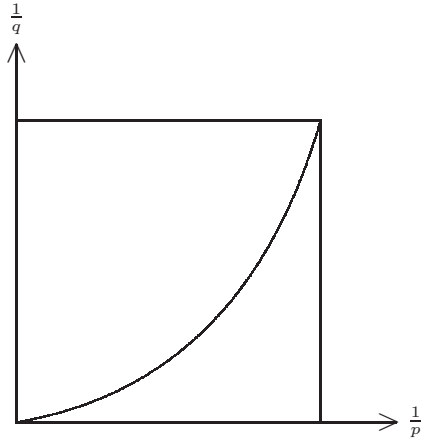

\section{Regular operators on subspaces of $L_p$}

It is well-known that in the case of regular operators (i.e., those which can be written as a difference of two positive operators) strict singularity is closer to compactness. More precisely, it was shown in \cite{CG} (see also \cite[Theorem 12]{HST17}) that for $1<q\leq p<\infty$, every strictly singular regular operator $T:L_p\rightarrow L_q$ must be compact. On the other hand, there exist simple examples of regular strictly singular operators $T:L_p\rightarrow L_q$ which are not compact when $p<q$.

In this section, we will explore this question for operators defined on a subspace of $L_p$. There is a natural definition of regularity which can be extended for operators defined on subspaces of a Banach lattice due to G. Pisier \cite{Pisier}: Given a subspace $X\subset L_p$, an operator $T:X\rightarrow L_q$ is \emph{regular} if there is $C\geq 0$ such that for all finite sequences $(x_i)_{i=1}^n\subset X$
\begin{equation}\label{eq:regular}
\Big\|\sup_{1\leq i\leq n} |Tx_i|\Big\|_q\leq C\Big\|\sup_{1\leq i\leq n} |x_i|\Big\|_p.
\end{equation}
We refer the reader to \cite{ST} for a more recent account on this and the closely related notions of $(p,q)$-regularity.

\begin{remark}
The proof of \cite[Theorem 5]{HST17} actually yields that if $p\geq 2\geq q$ every strictly singular operator from a subspace $X$ of $L_p$ into $L_q$ must be compact (regularity is not even necessary in this case).
\end{remark}

\begin{proposition}
Let $2<q<p$ and $X$ a subspace of $L_p$. If $T:X\rightarrow L_q$ is regular and strictly singular, then $T$ is compact.
\end{proposition}

\begin{proof}
Let $T:X\rightarrow L_q$ be a regular and strictly singular operator, and suppose that $T$ is not compact. Hence, there exists a normalized weakly null sequence $(x_n)\subset X$ such that $\|Tx_n\|\geq \alpha>0$ for every $n\in\mathbb N$. Using the Kadec-Pelczynski alternative \cite{KP} both for $(x_n)$ and $(Tx_n)$, together with the fact that $2<q<p$, it follows that necessarily $(x_n)$ must be equivalent to the unit vector basis of $\ell_2$ while $(Tx_n)$ must be equivalent to the unit vector basis of $\ell_q$. In particular, without loss of generality we can assume $(Tx_n)$ are pairwise disjoint and then using regularity of $T$ we get
\begin{align*}
\alpha n^{1/q}&\leq \Big\|\Big(\sum_{i=1}^n|Tx_i|^q\Big)^{1/q}\Big\|_q = \Big\|\sup_{1\leq i\leq n}|Tx_i|\Big\|_q \\
&\leq C \Big\|\sup_{1\leq i\leq n}|x_i|\Big\|_p \leq  C\Big\|\Big(\sum_{i=1}^n|x_i|^p\Big)^{1/p}\Big\|_p\\
&\leq C n^{1/p},
\end{align*}
which is a contradiction for large $n$.
\end{proof}

We will see next by means of particular examples that in all remaining cases, \cite[Theorem 12]{HST17} cannot be extended for operators defined on an arbitrary subspace of $L_p$.

\begin{proposition}
Given $q\leq p\leq s<2$, there is a subspace $X\subset L_p$ isomorphic to $\ell_s$ and a regular strictly singular operator $T:X\rightarrow L_q$ which is not compact.
\end{proposition}

\begin{proof}
If $p<s<2$, then let $(g_n)$ be a sequence of independend identically distributed $s$-stable random variables supported on $[0,\frac12]$, while if $s=p$, then let $(g_n)$ be a normalized pairwise disjoint sequence supported on $[0,\frac12]$. Let $(r_n)$ be a sequence of Rademacher random variables supported on $[\frac12,1]$. Let $x_n=g_n+r_n$. Since $(g_n)$ are equivalent to the unit vector basis of $\ell_s$, and $(r_n)$ are equivalent to the unit vector basis of $\ell_2$ we have
\begin{align*}
\Big\| \sum_{n=1}^m a_n x_n\Big\|_p &=\Big( \Big\| \sum_{n=1}^m a_n g_n\Big\|_p^p+\Big\| \sum_{n=1}^m a_n r_n\Big\|_p^p\Big)^{1/p}\\
&\approx \max\Big\{ \Big\| \sum_{n=1}^m a_n g_n\Big\|_p, \Big\| \sum_{n=1}^m a_n r_n\Big\|_p\Big\}\\
&\approx \Big(\sum_{n=1}^m |a_n|^s\Big)^{1/s}.
\end{align*}
Here, the last equivalence follows from the fact that $s<2$. Let $X$ be the closed linear span of $(x_n)$ in $L_p$, which is isomorphic to $\ell_s$.

Let now consider the operator $T:X\rightarrow L_q$ given by $Tf=f\chi_{[\frac12,1]}$. Clearly, $T$ is a regular operator. Moreover, as $Tx_n=r_n$, where $(r_n)$ is equivalent to the unit vector basis of $\ell_2$ and $s<2$, it follows that $T$ is strictly singular but not compact.
\end{proof}

\begin{proposition}
Given $p>2$, there is a subspace $X\subset L_p$ isomorphic to $\ell_2$ and a regular strictly singular operator $T:X\rightarrow L_p$ which is not compact.
\end{proposition}

\begin{proof}
Let $(h_n)$ be a sequence of normalized pairwise disjoint functions in $L_p$ whose support is contained in $[0,\frac12]$. Let $(r_n)$ be a sequence of Rademacher random variables supported on $[\frac12,1]$. Let $x_n=h_n+r_n$. Since $(h_n)$ are equivalent to the unit vector basis of $\ell_p$, and $(r_n)$ are equivalent to the unit vector basis of $\ell_2$ we have
\begin{align*}
\Big\| \sum_{n=1}^m a_n x_n\Big\|_p &=\Big( \Big\| \sum_{n=1}^m a_n h_n\Big\|_p^p+\Big\| \sum_{n=1}^m a_n r_n\Big\|_p^p\Big)^{1/p}\\
&\approx \max\Big\{ \Big\| \sum_{n=1}^m a_n h_n\Big\|_p, \Big\| \sum_{n=1}^m a_n r_n\Big\|_p \Big\}\\
&\approx \Big(\sum_{n=1}^m a_n^2\Big)^{1/2}.
\end{align*}
Here, the last equivalence follows from the fact that $p>2$. Let $X$ be the closed linear span of $(x_n)$ in $L_p$, which is isomorphic to $\ell_2$.

Let now consider the operator $T:X\rightarrow L_p$ given by $Tf=f\chi_{[0,\frac12]}$. Clearly, $T$ is a regular operator. Now, since $Tx_n=h_n$, with $(h_n)$ equivalent to the unit vector basis of $\ell_p$ and $p>2$, it follows that $T$ is strictly singular but not compact.
\end{proof}

\end{document}